\documentclass[11pt]{amsart}
\usepackage{amssymb,amstext,amsmath,amscd,amsthm,amsfonts,enumerate,graphicx,latexsym}
\usepackage{extarrows}
\newtheorem{theorem}{Theorem}[section]
\newtheorem{lemma}[theorem]{Lemma}
\newtheorem{corollary}[theorem]{Corollary}

\theoremstyle{definition}
\newtheorem{definition}[theorem]{Definition}

\newtheorem{example}[theorem]{Example}
\newtheorem{remark}[theorem]{Remark}
\newtheorem{conjecture}[theorem]{Conjecture}

\numberwithin{equation}{section}
 
\title[Geometry of varieties for graded maximal Cohen--Macaulay modules]{Geometry of varieties for graded maximal Cohen--Macaulay modules}
\keywords{graded maximal Cohen--Macaulay modules, graded countable Cohen--Macaulay representation type, module variety.}
\subjclass[2010]{Primary  13C14;  Secondary 16G60.}
\dedicatory{Dedicated to Professor Yuji Yoshino on the occasion of his retirement.} 
\thanks{The author was supported by JSPS KAKENHI Grant Number 18K13399.}
\author{Naoya Hiramatsu} 
\address{General Education Program, National Institute of Technology, Kure College, 2-2-11, Agaminami, Kure Hiroshima, 737-8506 Japan}
\email{hiramatsu@kure-nct.ac.jp} 

\begin{document}
\def\Ank{{\mathbb{A}^n (k)}}
\def\CM{{\mathrm{CM}}}
\def\Coker{\mathrm{Coker}}
\def\depth{\mathrm{depth}}
\def\End{\mathrm{End}}
\def\grEnd{{}^{\ast}\mathrm{End}}
\def\Hom{\mathrm{Hom}}
\def\grHom{{}^{\ast}\mathrm{Hom}}
\def\Im{\mathrm{Im}}
\def\Ker{\mathrm{Ker}}
\def\m{\mathfrak m}
\def\Mod{\mathrm{Mod}}
\def\mod{\mathrm{mod}}
\def\grmodR{\mathrm{mod}^{\Z}(R)}
\def\md{\mathrm{Mod}_d ^A (k)}
\def\OO{\mathcal O} 
\def\p{\mathfrak p}
\def\q{\mathfrak q}
\def\RepS{\mathrm{Rep}_S }
\def\RepSRV{\mathrm{Rep}_S (R, V)}
\def\Z{\mathbb Z}
\maketitle


\begin{abstract}
We study a variety for graded maximal Cohen--Macaulay modules, which was introduced by Dao and Shipman.     
The main result of this paper asserts that there are only a finite number of isomorphism classes of graded maximal Cohen--Macaulay modules with fixed Hilbert series over Cohen--Macaulay algebras of graded countable representation type.    
\end{abstract}

\section{Introduction}
Let $k$ be an algebraically closed field and $A$ a finite-dimensional $k$-algebra.
Denote by $\md$ the set of $A$-module structures on $k^d$.
Then $\md$ is an algebraic set, which is called the {\em module variety} of $d$-dimensional $A$-modules. 
It has been studied very successfully by many authors including Bongartz \cite{B}, Gabriel \cite{G74} and Morrison \cite{M80}.

Let $R = \oplus _{i=0}^{\infty} R_i$ be a commutative positively graded affine $k$-algebra with $R_0 = k$. 
Let $S$ be a graded Noetherian normalization of $R$. 
Then $R$ is a finitely generated graded $S$-module. 
A finitely generated graded $R$-module $M$ is said to be maximal Cohen--Macaulay if $M$ is graded free as a graded $S$-module. 
Then ${}_S M \cong {}_S S \otimes _k V$ for some finite dimensional graded $k$-vector space $V$. 
Dao and Shipman \cite{DS17} introduce the representation scheme $\RepSRV$ for graded maximal Cohen--Macaulay $R$-modules and explore it. 
The scheme gives the analogy of a module variety for finitely generated modules over a finite dimensional algebra.

First we shall show the following theorem, which gives an analogy of a result of Bongartz \cite{B}.

\begin{theorem}\label{Main 2}[Theorem \ref{morita}]
Let $A$ and $B$ be graded Cohen--Macaulay $k$-algebras with $\dim _k A_i = \dim _k B_i$ for all $i$. 
Suppose that $B$ is an integral domain and Gorenstein on the punctured spectrum. 
Then the following conditions are equivalent:
\begin{itemize}
\item[(a)] $A$ and $B$ are isomorphic graded $S$-algebras. 

\item[(b)] For a certain finite dimensional graded $k$-vector space $V$, $\RepS (A, V)(k)$ and $\RepS (B, V)(k)$ are $\grEnd _S (S \otimes _k V)$-equivariantly isomorphic. 
\end{itemize} 
\end{theorem}

We say that a graded Cohen--Macaulay ring $R$ is of graded countable Cohen--Macaulay representation type if there are countably many isomorphism classes of indecomposable graded maximal Cohen--Macaulay $R$-modules up to shift.
In Section \ref{countable}, we focus on a graded Cohen--Macaulay algebra which is of countable Cohen--Macaulay representation type. 
We shall give the following finiteness result about graded maximal Cohen--Macaulay modules.

\begin{theorem}[Theorem \ref{B}]\label{Main}
Let $k$ be an algebraically closed uncountable field and let $R$ be of graded countable Cohen--Macaulay representation type. 
For each finite dimensional $k$-vector space $V$, there are finitely many isomorphism classes of maximal Cohen--Macaulay $R$-modules which are isomorphic to $S \otimes _k V$ as graded $S$-modules. 
In other words, there are only a finite number of isomorphism classes of maximal Cohen--Macaulay $R$-modules with fixed Hilbert series. 
\end{theorem}

As an application, we show that the notion of graded countable Cohen--Macaulay representation type coincides with the notion of graded  discrete Cohen--Macaulay representation type in the sense of Drozd and Tovpyha \cite{DT14} provided that the algebra has at most an isolated singularity (Corollary \ref{discrete}).

\section{A variety for graded maximal Cohen--Macaulay modules}\label{variety}

Throughout the paper, $k$ is an algebraically closed field of characteristic $0$ and all algebras $R$ are commutative positively graded affine $k$-algebras. 
That is, $\displaystyle R = \oplus _{i=0}^{\infty} R_i$ with $R_0 = k$ and $R$ is finitely generated. 
We denote the unique maximal ideal by $R_{+} = \oplus _{i>0}^{\infty} R_i$. 
Then we can take a graded Noetherian normalization $S$ of $R$. 
That is, $S = k[y_1, ..., y_n ] \subseteq R$ where $n= \dim R$ such that $R$ is a finitely generated graded $S$-module (see \cite[Theorem 1.5.17]{BH}). 
Let $M$ be a finitely generated $\Z$-graded $R$-module.  
For $i \in \Z$, $M(i)$ is defined by $M(i)_n = M _{n+i}$. 
We denote by $\Hom_R (M, N) _i$ consisting of homogenous morphisms of degree $i$ and we set 
$$
\grHom _R (M, N) = \oplus _{i \in \Z} \Hom _R (M, N) _i 
$$
for graded $R$-modules $M$ and $N$.

\begin{definition}
A finitely generated graded $R$-module $M$ is said to be maximal Cohen--Macaulay (abbr. MCM) if $M$ is graded free as an $S$-module: 
$$
M \cong \bigoplus ^m _{i=1} S (l_i). 
$$
And then $M \cong S \otimes _k V$ for some finite dimensional graded $k$-vector space $V$, so that we call it a graded MCM $R$-module of type $V$. 
\end{definition}

\begin{remark}\label{Hilbert series}
Let $V$ be a finite dimensional graded $k$-vector space and $M$ and $N$ be graded MCM $R$-modules of type $V$. 
Then one can see that $\dim _k M_i = \dim _k N_i $ for all $i$. 
Hence $M$ and $N$ have the same Hilbert series. 
\end{remark}

We recall the notion of a variety for graded MCM modules, which plays a key role in our results. 
For details we recommend the reader to refer to \cite{DS17}.

Given a graded MCM $R$-module $M$, since $M \cong S \otimes _k V$, there exists a degree $0$ graded $S$-algebra homomorphism $\mu : R \to \grEnd _S (S \otimes _k V)$, which is called a matrix representation of $M$: 
$$
\mu \in \Hom _{S-alg} (R, \grEnd _S (S \otimes _k V))_0. 
$$
Then $\Hom _{S-alg} (R, \grEnd _S (S \otimes _k V))_0$ is an algebraic variety over $k$. 
We denote it by $\RepSRV (k)$.

\begin{example}
Let $R= k[x, y] /(x^2)$ with $\deg x = \deg y =1$ and $S= k[y]$ a graded Noetherian normalization of $R$. 
Set $V = k \oplus k(-1)$. 
Then giving a graded MCM $R$-module which is isomorphic to $S \otimes _k V = S \oplus S(-1)$ is equivalent to giving a $\mu  \in \End _S (S \otimes _k V)_1$ with $\mu ^2 = 0$. 
Note that 
$$
\End _S (S \otimes _k V)_1 = \left \{
\mu (x) = \left(\begin{smallmatrix}
ay & by^2 \\
c & dy
\end{smallmatrix}
\right)
|
a, b, c, d \in k
\right\}. 
$$  
Hence one can show that 
$$
\RepSRV (k) 
= \mathit{V}(a^2 + bc, ab+bd, ac + bc, bc + d^2) \subset \mathbb{A}_k ^{4}. 
$$
\end{example}

\begin{remark}\label{DS}
Dao and Shipman \cite{DS17} introduce a functor $\RepSRV$ from the category of commutative $k$-algebras to sets. 
For a commutative $k$-algebra $T$, they define the notion of $T$-flat family of $V$-framed $R$-modules (\cite[Definition 2.1]{DS17}), and $\RepSRV (T)$ is the set of these modules. 
They show that $\RepSRV$ is represented by an affine variety of finite type (\cite[Proposition 2.2]{DS17}). 
\end{remark}

\begin{remark}\label{F}
\begin{itemize}

\item[(1)] The algebraic group $G_V = \mathrm{Aut} _S (S \otimes _k V )_0$ acts on $\RepSRV (k)$ by conjugation, and we have 1-1 correspondence;
$$
\begin{array}{l}
\left\{ \text{$G_V$-orbits in $\RepSRV (k)$} \right \} 
\xleftrightarrow{1-1} \left\{ M \ | \ M \cong S \otimes_k V \text{as graded $S$-modules } \right \}/\cong .
\end{array}
$$

\item[(2)] Note that $\RepSRV (k)$ parameterizes graded MCM $R$-modules with fixed Hilbert series and $S$-module structure (see Remark \ref{Hilbert series}).  
\end{itemize}
\end{remark}

Let $X$ be a $G$-variety. 
Namely $X$ is a variety equipped with an action of the group $G$.  
For $x \in X$, we denote by $\OO (x)$ the $G$-orbit of $x$.

\begin{example}
Let $R= k[x, y] /(x^2)$ with $\deg x = \deg y =1$ and $V = k \oplus k(-1)$. 
Then $\RepSRV (k)$ contains the following three $G_V$-orbits:  
$$
R \cong \OO \left( \begin{pmatrix} 0 &0 \\ 1&0\end{pmatrix} \right),  
(x, y^2)R(1) \cong \OO \left( \begin{pmatrix} 0 &y^2 \\ 0&0\end{pmatrix}\right), 
R/(x) (1) \oplus R/(x)  \cong \OO \left( \begin{pmatrix} 0 &0 \\ 0&0\end{pmatrix}\right). 
$$
\end{example}

The next lemma follows from the definition. 
See \cite[10.6.5]{TY05} for instance.

\begin{lemma}\label{morphism}
Let $\mu : R \to \Hom _{S-alg} (R, \End _S (S \otimes _k V))_0$ and $\nu : R \to \Hom _{S-alg} (R, \End _S (S \otimes _k W))_0$ be matrix representations. 
A degree $i$ graded $S$-homomorphism $\alpha _i  \in \Hom_S (S\otimes _k V, S\otimes _k W)_i$ is a graded $R$-homomorphism if and only if $\alpha _i \circ \mu (r) = \nu (r) \circ \alpha _i $ for each homogenous element $r \in R$. 
\end{lemma}

For graded CM $k$-algebras $A$ and $B$, we say that a morphism of algebraic varieties $\Phi : \RepS (A, V)(k) \to \RepS (B, V)(k)$ is in an $\grEnd _S (S \otimes _k V)$-equivariant if $\alpha _i \circ \mu = \mu \circ \alpha _i$ implies $\alpha _i \circ \Phi (\mu ) = \Phi (\mu) \circ \alpha _i$ for each $\mu \in \RepS (A, V)(k)$, $\alpha _i \in \End _S (S \otimes _k V)_i$ and $i \in \Z$. 
We also say that a graded CM $k$-algebra $R$ is Gorenstein on the punctured spectrum if each graded localization $R_{(\p )}$ is Gorenstein for each graded prime ideal $\p$ with $\p \not= R_{+}$.

Bongartz \cite{B} studied how much the knowledge of module varieties (of finite dimensional modules) reflect the structure of the (finite dimensional) algebras. 
Here we have the analogous result.

\begin{theorem}\label{morita}
Let $A$ and $B$ be graded CM $k$-algebras with $A$, $B \cong S \otimes_k V$. 
Suppose that $B$ is an integral domain and Gorenstein on the punctured spectrum. 
Then the following conditions are equivalent: 

\begin{itemize}
\item[(a)] $A$ and $B$ are isomorphic graded $S$-algebras. 

\item[(b)] For the finite dimensional graded $k$-vector space $V$, the algebraic sets $\RepS (A, V)(k)$ and $\RepS (B, V)(k)$ are isomorphic in an $\grEnd _S (S \otimes _k V)$-equivariant. 

\end{itemize} 
\end{theorem}

\begin{proof}
We show that (b) implies (a). 
Let $\mu \in \RepS (A, V)(k)$ be the matrix representation of $A$ and $Y$ the graded MCM $B$-module which corresponds to $\Phi ( \mu)$. 
From the knowledge of matrix representations (Lemma \ref{morphism}) and the assumption of $\Phi$, we have $S$-algebra isomorphisms $A \cong \grEnd_A ( A) \cong \grEnd_B (Y )$. 
Since $B$ is Gorenstein on the punctured spectrum, $Y$ is torsion free (cf. \cite[Proposition 12.8]{LW12}). 
Moreover $B$ is an integral domain, thus the natural graded $S$-algebra morphism which sends $b \in B$ to the multiplication mapping $b_Y$ by $b$ on $Y$  
$$
B \to {}^{\ast} \End_B ( Y ); \quad b \mapsto b_Y 
$$
is injective. 
Hence we have an injective graded $S$-algebra morphism $B \to A$. 
Since $A$ and $B$ have the same type $V$, that is, $\dim _k A_i = \dim _k B_i$ for each $i$, the morphism must be isomorphism. 
Therefore $A$ is isomorphic to $B$ as a graded $S$-algebra. 
\end{proof}

\begin{remark}
In the result due to Bongartz \cite[Theorem 1]{B}, it is assumed that the isomorphism of the module varieties is in a $\mathrm{GL}_d$-equivariant. 
Under the notation in Theorem \ref{morita},  if we assume that the isomorphism is in a $G_V$-equivariant, we only have the isomorphism $A_0 \cong \End_A (A)_0 \cong \End _B(Y)_0$. 
That is why we assume that the isomorphism in an $\grEnd _S (S \otimes _k V)$-equivariant instead of in a $G_V$-equivariant. 
\end{remark}

The assumption on $B$ is given only to prove the theorem. 
At the end of this section, we state the conjecture.

\begin{conjecture}
Theorem \ref{morita} would be  valid without any assumptions on algebras $A$ and $B$ except that $A$ and $B$ are isomorphic as graded $S$-modules. 
\end{conjecture}

\section{On graded countable Cohen--Macaulay representation type}\label{countable}

In what follows, we always assume that $k$ is an algebraically closed uncountable field. 
In this section we focus on graded CM $k$-algebras which are of graded countable CM representation type. 
Note that graded countable type includes graded finite type, in which case there are only finitely many isomorphism classes of indecomposable graded MCM modules up to shift. 
To avoid confusion, we make the following definition.

\begin{definition}
We say that $R$ is of graded countable {\it{strictly}} CM representation type if there are {\it infinitely} but only countably many isomorphism classes of indecomposable graded MCM modules up to shift. 
\end{definition}

\begin{example}
Let $R = k[x, y]/(x^2)$ with $\deg x = \deg y = 1$. 
It is known that $R$ is of graded strictly countable CM representation type whose indecomposable graded MCM $R$-modules are $I_n  = (x, y^n)R$ for $n\geq 0$ up to shift (cf \cite{BGS87}). 
The graded Noetherian normalization of $R$ is $S = k[y]$. 
Set $V = k(-1) \oplus k(-n)$.
Then $I_n \cong S \otimes _k V$ as $S$-modules. 
\end{example}

We state our main theorem of this paper.

\begin{theorem}\label{main}
Let $R$ be of graded countable CM representation type. 
For each finite dimensional graded $k$-vector space $V$, there are finitely many isomorphism classes of graded MCM $R$-modules of type $V$.
\end{theorem}

It is natural  to ask what happens if we fix the Hilbert polynomial instead of the Hilbert series. 
We have the following example.

\begin{example}
Let $R = k[x, y]/(x^2)$ with $\deg x = \deg y = 1$ and $I_n = (x, y^n)R$. 
Then Hilbert polynomials of $I_n$ are $2$ for all $n$ and $I_n \not\cong I_m$ if $n \not= m$. 
\end{example}

Now let us prove Theorem \ref{main}.  
First we mention a lemma. 

\begin{lemma}\label{A}
Let $X \subseteq \Ank$ be an algebraic set and let $X_i \subsetneq X$ be closed subsets with $\dim X_i \lneq \dim X$. 
Then $X$ can be never represented by a countable union of $X_i$.  
\end{lemma}

\begin{proof}
We prove by induction on $n$. 
Suppose that $n = 1$. 
Then $\dim X_i = 0$, so that $X_i$ is a finite set of points. 
Assume that  $X = \cup_{i\geq1}X_i$, and then $X$ contains infinitely but countably many points. 
This is a contradiction. 
Since $k$ is an uncountable filed $X$ must contain uncountably many points. 
Suppose that $n \geq  2$. 
Considering the irreducible decomposition of $X$, we may assume that $X$ is irreducible. 
Then $X$  is represented by $V (\p)$ for some prime ideal $\p$ in $k[ x_1. x_2, \cdots , x_n]$. 
We also put $I_i$ ideals with $X_i = V (I_i)$. 
After renumbering, we may assume that $X$ is not contained in $V (x_1 - c)$ for all $c \in k$. 
For each minimal prime ideal $\q$ of $I_i$, there are a finite number of elements $c \in k$ such that $x_1 - c \in \q$. 
Note that a number of minimal prime ideals are finite. 
Recall that $k$ is an uncountable field. 
Thus we can take $\lambda \in k$ such that $x_1 - \lambda$ is neither contained in $\p$ nor all minimal prime ideals of $I_i$ for all $i$. 
Consider the quotient by $x_1 - \lambda$. 
Then we can reduce the case that $X=\cup_{i\geq1}X_i$ is a closed subset of $\mathbb{A}^{n-1}_k$ preserving $\dim X_i \lneq \dim X$ for all $i$. 
By the induction hypothesis, we obtain the assertion. 
\end{proof}

\begin{theorem}\label{B}
Let $X\subseteq \Ank$ be an algebraic set. 
Suppose that $X$ is represented by a countable disjoint union of locally closed subsets, that is  $X = \cup _{i \geq 1} Y_i$ where $Y_i$ are locally closed and $Y_i \cap Y_j = \varnothing$ for $i \not= j$. 
Then it is a ``finite" union. 
\end{theorem}

\begin{proof}
Let $X_1$, $X_2$, ... $X_m$ be irreducible components of $X$. 
For each component $X_k$, we have
$$
X_k = X_k \cap X = X_k \cap (\cup _{i \geq 1} Y_i )= \cup _{i \geq 1} (X_k \cap Y_i ). 
$$
Note that $X_k \cap Y_i $ are locally closed for all $i$. 
By Lemma \ref{A}, there exists $j$ such that $\dim X_k = \dim X_k \cap Y_j$, so that $X_k = \overline{X_k \cap Y_j}$. 
Since $X_k \cap Y_j$ is open in $\overline{X_k \cap Y_j} = X_k$, 
$$
X_{k, 2} := X_k \backslash (X_k \cap Y_j)
$$ 
is closed and $\dim X_{k, 2} \lneq \dim X_k$. 
We decompose $X_{k, 2}$ into its irreducible components $X_{k, 2, 1}$, $X_{k, 2, 2}$, ... $X_{k, 2, m'}$ and apply the above argument for each components $X_{k, 2, j}$. 
Then we also obtain the closed subsets $X_{k, 3, j'}$ with $\dim X_{k, 3, j'} \lneq \dim X_{k, 2, j'}$ for all $j'$.   
Repeating this at most $\dim X$ times, we achieve the case that each components are of dimension $0$. 
Since a closed subset of dimension $0$ is a finite set of points, these subset can be represented by a finite union of (locally) closed subsets which derive from $Y_i$. 
A number of $Y_i$ which are appeared in this arguments is finite, so that we obtain the assertion.  
\end{proof}
 

\begin{remark}\label{E} 
Let $X$ be a $G$-variety and $x \in X$. 
Then one can show the following statements hold. 
\begin{itemize}
\item[(1)] $\OO (x)$ is locally closed. See \cite[Proposition 21.4.3(i)]{TY05}.  
\item[(2)] $\OO (x) = \OO (y)$ if and only if $\overline{\OO (x) } = \overline{\OO (y)}$ for $x, y \in X$. See \cite[Proposition 3.5]{K82} for instance. 
\end{itemize}
\end{remark}

The proof of Theorem \ref{main} is based on our Theorem \ref{B}.

\medskip
\noindent
{\it Proof of Theorem \ref{main}.}
According to Remark \ref{F}, it is enough to show that $\RepSRV (k)$ consists of finitely many orbits. 
If $R$ is of graded finite representation type,  the assertion is clear. 
Suppose that $R$ is of graded strictly countable CM representation type. 
Then $\RepSRV (k)$ consists of infinitely but countably many orbits. 
It follows from Theorem \ref{B} and Remark \ref{E} that $\RepSRV (k)$ is a finite union of orbits. 
\qed
\medskip

\section{An application}\label{application}

In the rest of this paper, we investigate the relation between graded countable CM representation type and graded discrete CM representation type. 
The following definition is taken from \cite{DT14}.

\begin{definition}\cite[Definition 1.1.]{DT14}
We say that $R$ is of graded {\it discrete} CM representation type if, for any fixed $r >0$, there are only finitely many isomorphism classes of indecomposable graded MCM $R$-modules with rank $r$ up to shift. 
Here the rank is taken over $S$.     
\end{definition}

One can show that if $R$ is of discrete CM representation type then $R$ is of countable CM representation type. 
Because there is only a countable set of graded MCM $R$-modules up to isomorphism and shift if $R$ is of discrete. 
The converse does not hold in general.

\begin{example}
Let $R = k[x, y]/(x^2)$ with $\deg x = \deg y = 1$ and $I_n  = (x, y^n)R$. 
It is known that $I_n$ is an indecomposable graded MCM $R$-modules for $n\geq 0$. 
Note that $\mathrm{rank} _S \  I_n  = 2$ since $I_n \cong S(-1) \oplus S(-n)$ where $S=k[y]$. 
But $I_n  \not\cong I_m $ if $n \not= m$, so that $R$ is not of graded discrete CM representation type.  
\end{example}

The following theorem is due to Dao and Shipman. 
We say that $R$ is with an isolated singularity if each graded localization $R_{(\p )}$ is regular for each graded prime ideal $\p$ with $\p \not= R_{+}$.

\begin{theorem}\cite[Theorem 3.1]{DS17}\label{D}
Assume that $R$ is with an isolated singularity. 
For each $r > 0$ there exists $\alpha _r >0$ such that if $M$ is an indecomposable graded MCM $R$-module with rank $r$ then 
$$
g_{max}(M) - g_{min}(M) < \alpha _r,  
$$
where $g_{max}(M) = max \{ m | (M / S_{+} M)m \not= 0 \}$ and $g_{min}(M) = min \{ m | (M / S_{+} M)m \not= 0 \}$. 
\end{theorem}

One can deduce from the theorem that for a MCM graded $R$-module of rank $r$ up to shift, only finitely many finite dimensional graded $k$-vector space $V$ come into question. 
Indeed, for an indecomposable graded MCM modules $M$ with rank $r$, we shift $M$ to normalize it, and then $g_{min}(M) = 0$. 
Thus $0 \leq g_{max}(M) < \alpha _r$. 
Note that $\alpha _r$ depends on only $r$. 
Since $M/S_{+}M \cong V$ as $k$-modules, we obtain the claim.

\begin{corollary}\label{discrete}
Let $R$ be of graded countable CM representation type. 
Suppose that $R$ is with an isolated singularity. 
Then $R$ is of graded discrete CM representation type. 
\end{corollary}

\begin{proof}
We identify an orbit of a point in $ \RepSRV (k)$ with the corresponding isomorphism classes of a graded MCM $R$-module (Remark \ref{F}(1)). 
For each $r>0$, we have the inclusion of sets:
\begin{multline*}
\{ M : \text{indecomposable graded MCM $R$-modules with rank $r$}\}/<\cong , \text{shift}> \\
\subseteq \displaystyle \bigcup _{V : \text{ finite dimensional graded } k\text{-vector space}}  \RepSRV (k). 
\end{multline*}
According to Theorem \ref{D}, the right hand side is a finite union of $ \RepSRV (k)$. 
Since $\RepSRV (k)$ consists of finitely many orbits (Theorem \ref{main}), the left hand side is a finite set.  
(Compare with \cite[Corollary A]{DS17}. See also \cite[Theorem 3.16]{Ka09}. )
\end{proof}

\begin{remark}
As mentioned in \cite[Definition 8.13.]{BD}, for a CM $k$-algebra $R$, the CM representation type of $R$ is said to be discrete if the set of isomorphism classes of indecomposable graded MCM modules is countable. 
Usually, we consider the notion in this sense. 
\end{remark}

\section*{Acknowledgments}
The author express his deepest gratitude to Yuji Yoshino for valuable discussions and helpful comments.  
The author also thank the referee for his/her careful reading and helpful comments that have improved the paper.
 
\ifx\undefined\bysame 
\newcommand{\bysame}{\leavevmode\hbox to3em{\hrulefill}\,} 
\fi

\end{document}